\documentclass{birkjour}
\usepackage{amsfonts}
\usepackage{amssymb}
\usepackage{ifthen}
\usepackage{graphicx}
\usepackage{color}
\newtheorem{thm}{Theorem}[section]
\newtheorem{cor}[thm]{Corollary}
\newtheorem{lem}[thm]{Lemma}

\theoremstyle{definition}
\newtheorem{defn}[thm]{Definition}
\theoremstyle{remark}

\newtheorem{conj}[thm]{Conjecture}
\newtheorem{ca}{Case}
\numberwithin{equation}{section}

\def\be{\begin{equation}}
\def\ee{\end{equation}}
\newcommand{\bthm}{\begin{thm}}
\newcommand{\ethm}{\end{thm}}
\newcommand{\bcor}{\begin{cor}}
\newcommand{\ecor}{\end{cor}}
\newcommand{\blem}{\begin{lem}}
\newcommand{\elem}{\end{lem}}
\newcommand{\bcon}{\begin{conj}}
\newcommand{\econ}{\end{conj}}
\newcommand{\bca}{\begin{ca}}
\newcommand{\eca}{\end{ca}}

\newcommand{\beq}{\begin{eqnarray}}
\newcommand{\beqq}{\begin{eqnarray*}}
\newcommand{\eeq}{\end{eqnarray}}
\newcommand{\eeqq}{\end{eqnarray*}}
\newcommand{\bdefe}{\begin{defn}}
\newcommand{\edefe}{\end{defn}}
\newcommand{\bpf}{\begin{proof}}
\newcommand{\epf}{\end{proof}}

\newtheorem{statement}{}[section]
\newtheorem{theoreme}[statement]{Theorem}
\newtheorem{lemme}[statement]{Lemma}
\newtheorem{remarque}[statement]{Remark}

\newtheorem{corollaire}[statement]{Corollary}

\newcommand\C{\mathbb C}
\newcommand\N{\mathbb N}

\newcommand\D{\mathbb D}

\newcommand\eps{\varepsilon}
\newcommand\ind{{\rm 1\kern-.30em I}}

\newcounter{alphabet}
\newcounter{tmp}
\newenvironment{Thm}[1][]{\refstepcounter{alphabet}%
\bigskip%
\noindent%
{\bf Theorem \Alph{alphabet}}%
{\bf .} \itshape}{\vskip 8pt}

\begin{document}

\title[Composition operator on the Hardy-Dirichlet space]{Estimate for norm of a composition operator on the Hardy-Dirichlet space}

\author[P. Muthukumar]{Perumal Muthukumar}
\address{Stat-Math Unit,
Indian Statistical Institute (ISI), Chennai Centre,
110, Nelson Manickam Road,
Aminjikarai, Chennai, 600 029, India.}
\email{pmuthumaths@gmail.com}

\author[S. Ponnusamy]{Saminathan Ponnusamy}

\address{Department of Mathematics,
Indian Institute of Technology Madras, Chennai-600 036, India}
\email{samy@iitm.ac.in, samy@isichennai.res.in }

\author[H. Queff\'elec]{Herv\'e Queff\'elec}
\address{Univ Lille Nord de France F-59,\kern 1mm 000 Lille,
USTL, Laboratoire Paul Painlev\'e U.M.R. CNRS 8524,
F-59\kern 1mm 655 Villeneuve D'ascq Cedex,
France.}
\email{herve.queffelec@univ-lille1.fr}

\thanks{The second  author is currently at ISI Chennai Centre.}

\date{December 29, 2017}

\begin{abstract}
By using  the Schur test, we give some upper and lower estimates on the norm of a composition operator on $\mathcal{H}^2$, the space of
Dirichlet series with square summable coefficients,  for the inducing symbol  $\varphi(s)=c_1+c_{q}q^{-s}$ where $q\geq 2$
is a fixed integer. We also give an estimate on the approximation numbers of such an operator.
\end{abstract}
\maketitle

\section{Introduction}
Let $\Omega$ be a domain in the complex plane $\C$. For a given analytic self map $\varphi$ of $\Omega$, the
corresponding composition operator $C_\varphi$ induced by the symbol $\varphi$ is defined by
$ C_\varphi(f)= f\circ\varphi  ~\mbox{ for every analytic function $f$ on $\Omega$}.$
In the classical case, $\Omega$ is taken as the unit disk $\D= \{z \in \C:\, |z|<1\}$ and the operator $C_\varphi$
is considered on various analytic function spaces on ${\mathbb D}$ such as the Hardy spaces $H^p$, the Bergman
spaces $A^p$ and the Bloch space $\mathcal B$.

For a real number $\theta$, we set $\C_\theta=\{s\in \C :\ {\rm Re}\,  s>\theta\}.$ In this article,
$\Omega$ will be taken to be the half plane $\C_{1/2}$, the map $\varphi$ to be the analogue of
affine map in the classical case and the composition operator $C_\varphi$ is considered on the Hardy-Dirichlet
space $\mathcal{H}^2$,  which is a Dirichlet series analogue of the classical Hardy space.

Determining the value of the norm of composition operators is not an easy task and hence, not much
is known on this problem even in the case of classical Hardy space except for some special cases. For example, the
norm of a composition operator on $H^2$ induced by the simple affine mapping of $\D$ is complicated
(see \cite[Theorem 3]{CO}). Not to speak of the approximation numbers of $C_\varphi$, even though the latter were
computed in \cite{CLDA}.  In case of the space  $\mathcal{H}^2$ of Dirichlet series with square-summable
coefficients, there are no good lower and upper bounds even for the norm of such operators except for some special cases.
As a first step, in this paper, we give some upper and lower estimates on the norm of a composition operator on $\mathcal{H}^2$,
for the inducing symbol $\varphi(s)=c_1+c_{q}q^{-s}$ with $q\in \N,\ \ q\geq 2$. Here $\N$ denotes the set
of all natural numbers and we set $\N_0=\N \cup\{0\}$. Without loss of generality, we will assume that $q=2$. One significant difference is that some properties of the Riemann zeta function, be it only  in the half-plane $\C_1$, are required.

The article is organized as follows. In Section \ref{MPQ3Sec2}, definition and some important properties of Hardy-Dirichlet
space $\mathcal{H}^2$ are recalled. Also, the boundedness of composition
operators on $\mathcal{H}^2$ is discussed. In Section \ref{MPQ3Sec3}, motivation for this work and estimates
for the norm of $C_\varphi$ for the affine-like inducing symbols are given.
Finally,  in Section \ref{MPQ3Sec4}, we give an estimate for approximation numbers of a composition operators
in our $\mathcal{H}^2$ setting.

One may refer to \cite{Zhu-livre} for basic information about analytic function spaces of $\D$ and operators on
them. Basic issues on composition operators on various function spaces on $\D$ may be obtained from \cite{COMA}. See
also \cite{HST} for results related to analytic number theory.

\section{Composition operators on the Hardy space of Dirichlet series}\label{MPQ3Sec2}

The Hardy-Dirichlet space $\mathcal{H}^2$ is defined by
\begin{equation}\label{hachede}
\mathcal{H}^2= \left\{ f(s)=\sum_{n=1}^\infty a_n n^{-s} :\, \Vert f\Vert^2 =\sum_{n=1}^\infty |a_n|^2<\infty \right\}.
\end{equation}
The space $\mathcal{H}^2$ has been used in \cite{HLS} for the study of completeness problems of a system of dilates
of a given function. The following properties are obvious:
\begin{itemize}
\item If $f\in \mathcal{H}^2$, then the Dirichlet series in (\ref{hachede}) converges absolutely in $\C_{1/2}$,
and therefore $\mathcal{H}^2$ is a Hilbert space of analytic functions on $\C_{1/2}$.
\item The functions $\{e_n\}$ defined on $\C_{1/2}$ by $e_{n}(s)=n^{-s}, n\geq 1$, form an orthonormal basis for $\mathcal{H}^2$.
\item Accordingly, the reproducing kernel $K_a$ of $\mathcal{H}^2$ ($f(a)=\langle f, K_a\rangle$
for all $f\in \mathcal{H}^2$) is given by
$$K_{a}(s)=\sum_{n=1}^\infty e_{n}(s)\overline{e_{n}(a)}=\zeta(s+\overline{a}), \hbox{\ with}\  a, s \in \C_{1/2},
$$
where $\zeta$ denotes the Riemann zeta function.
\end{itemize}

Let $H(\Omega)$ denote the space of all analytic functions defined on $\Omega$. If $\varphi:\,\C_{1/2}\to \C_{1/2}$
is analytic, then the composition operator
$$C_\varphi: \,\mathcal{H}^2\to H( \C_{1/2}),\quad C_{\varphi}(f)=f\circ \varphi,
$$
is well defined and we wish to know for which ``symbols" $\varphi$ this operator maps $\mathcal{H}^2$ to itself. Then, $C_\varphi$ is a bounded linear operator on $\mathcal{H}^2$ by the closed graph theorem. A complete answer to this fairly delicate question was obtained in \cite{GOHE}. A slightly improved version of the same may be stated in the following form, as far as uniform convergence on all half-planes $\C_\eps$ is concerned. See \cite{QSE} for details.

\begin{Thm}\label{gh} The analytic function $\varphi:\,\C_{1/2}\to  \C_{1/2}$ induces a bounded composition operator on  $ \mathcal{H}^2$ if and only if
\begin{equation}\label{phi-expn}
\varphi(s)=c_{0}s+\sum_{n=1}^\infty c_n n^{-s}=: c_{0}s+\psi(s),
\end{equation}
where $c_0\in \N_0$ and the Dirichlet series $\sum\limits_{n=1}^\infty c_n n^{-s}$ converges uniformly in each half-plane $\C_\eps,\ \eps>0$.
 Moreover, $\psi$ has the following mapping properties:
\begin{enumerate}
\item If $c_0\geq 1$, then $\psi(\C_0)\subset \C_0 \hbox{\ and so}\ \varphi(\C_0)\subset \C_0$.
\item If $c_0=0$, then $\psi(\C_0)=\varphi(\C_0)\subset \C_{1/2}$.
\end{enumerate}
\end{Thm}


In addition to the above formulation, it is worth to mention that $\Vert C_\varphi\Vert\geq 1$ and
$$\Vert C_\varphi\Vert=1 \Longleftrightarrow   c_0\geq 1.
$$
This result follows easily from the fact that $C_\varphi$ is contractive on $ \mathcal{H}^2$ if $c_0\geq 1$ (See \cite{GOHE}).

\section{A special, but interesting case}\label{MPQ3Sec3}
 To our knowledge, except the recent work of Brevig \cite{BR} in a slightly different context, no result has appeared in the literature on sharp evaluations of the norm of $C_\varphi$ when $c_0=0$. The purpose of this work is to make some attempt, in the apparently simple-minded case
\begin{equation}\label{aff}
\varphi(s)=c_1+c_{2}2^{-s} \hbox{\ with}\ {\rm Re}\,  c_1\geq \frac{1}{2}+|c_2|.
\end{equation}
The condition on $c_1$ and $c_2$ in (\ref{aff}) is the exact translation of the mapping conditions of ``affine map" to be a map of $\C_0$ into $\C_{1/2}$.

We should point out the fact that, even though the symbol $\varphi$ is very simple, the boundedness of $C_\varphi$, and its norm, are far from being clear. This is already the case for affine maps $\varphi(z)=az+b$ from $\D\to \D$ whose exact norm has a complicated expression first obtained by Cowen \cite{CO} and then by the third-named author of this article (see \cite{HE}) with a simpler approach based on an adequate use of the Schur test, which we recall  in Lemma \ref{schur} below, under an adapted form.

Finally, we would like to mention the following: In \cite{HU}, Hurst obtained the norm of $C_\varphi$ on weighted Bergman spaces for the affine symbols whereas in \cite{HAM}, Hammond obtained a representation for the norm of $C_\varphi$ on the Dirichlet space for such affine symbols.

\begin{lemme}\cite[page~24]{HA}\label{schur}
Let $A=(a_{i,j})_{i\geq 0, j\geq 1}$ be a scalar matrix, formally defining a linear map $A:\,\ell^{2}(\N)\to \ell^{2}(\N_0)$ by the formula
$A(x)=y \hbox{\ with}\ y_i=\sum\limits_{j=1}^{\infty} a_{i,j} x_j$.  Assume that there exist two positive numbers $\alpha$ and $\beta$ and  two sequences $(p_i)_{i\geq 0}$ and $(q_j)_{j\geq 1}$ of positive numbers such that
\begin{equation}\label{a}
\sum\limits_{i=0}^{\infty} |a_{i,j}|q_i\leq \alpha p_j ~\mbox{ for all }~ j\geq 1
\end{equation}
and
\begin{equation}\label{b}
\sum\limits_{j=1}^{\infty} |a_{i,j}|p_j\leq \beta q_i  ~\mbox{ for all }~ i\geq 0.
\end{equation}
Then
$\Vert A\Vert\leq \sqrt{\alpha \beta}.$
\end{lemme}
\begin{remarque}\label{spec-radius}
Let $\varphi$ be a map as in \eqref{aff}. Then $C_\varphi$ is compact operator on $ \mathcal{H}^2$ if and only if ${\rm Re}\,  c_1 > \frac{1}{2}+|c_2|$
(see \cite[Corollary 3]{BA}). Also the spectrum of $C_\varphi$ is
$$\sigma(C_\varphi)= \{0,1\} \cup \{[\varphi '(\alpha)]^k:\, k\in \N \},
$$
where $\alpha$ is the fixed point in $\C_{1/2}$ (see \cite[Theorem 4]{BA}). Since the spectrum  $\sigma(C_\varphi)$ is compact, we have $|\varphi '(\alpha)|< 1$ and thus the spectral radius
$$r(C_\varphi):= \sup \{|\lambda| :\, \lambda \in \sigma(C_\varphi) \}
$$
is equal to $1$.
\end{remarque}

In \cite{Hed}, Hedenmalm asked for estimate from above for the norm $\|C_\varphi\|$ in terms of $\varphi(+\infty)$, that is, $c_1$ for the map $\varphi(s)=c_1+c_{2}2^{-s}$. We give a partial answer to his question at least for this special choice of $\varphi$.
To do this, we list below some useful lemmas here.

\begin{lemme}\label{rely} Let $s>1$. Then, we have
 $$ \frac{1}{s-1}\leq \zeta(s)\leq  \frac{s}{s-1}.
 $$
\end{lemme}
 \begin{proof}
 The result follows, by comparison with an integral, from the fact that $x\mapsto x^{-s}$ is decreasing for $s>1$. 
 See for instance, \cite[p.~299]{samy}. Indeed for $f(x)=x^{-s}=e^{-s \ln x}$, we have
 $$
 \int_{1}^{\infty}f(x)dx \leq \sum_{k=1}^\infty f(k) \leq f(1)+\int_{1}^{\infty}f(x)dx,
 $$
 from which one can obtain the desired inequality, since
 $\int_{1}^{\infty}f(x)dx=\frac{1}{s-1}$.
 \end{proof}

\begin{lemme}\label{zetabound}
For all $ s>1$, we have
\begin{equation}\label{eq-lem1}
\frac{1}{s-1}+\left(\frac{s-1}{s}\right)\frac{1}{\sqrt{2\pi}}\leq \zeta(s) .
\end{equation}
\end{lemme}
\begin{proof}
Let
$$h(s)=\frac{1}{s-1}+\left(\frac{s-1}{s}\right)\frac{1}{\sqrt{2\pi}}.
$$
Then, we observe that  both $h$ and $\zeta $ are decreasing functions on $(1,\infty).$
Thus,
$$h(s)\leq h(3)=\frac{1}{2}+\frac{1}{3}\sqrt{\frac{2}{\pi}}<\frac{1}{2}+\frac{1}{3}< 1<\zeta(s) ~\mbox{ for all $s\geq 3$}.
$$
This shows that the inequality \eqref{eq-lem1} is true for $s\geq 3$.
Now we need to verify the inequality \eqref{eq-lem1} only for $1<s<3$. By setting $s=x+1$,
it is enough to prove that
$$h(x+1)=\frac{1}{x}+f(x)\leq \zeta(x+1) \mbox{~for~} 0<x<2,
$$
where
$$f(x)=\frac{1}{\sqrt{2\pi}}\left(\frac{x}{x+1}\right).
$$
Clearly, $f$ is an increasing function on $x>0$. From \cite[Lemma 10]{BR}, we have
 $$\frac{1}{x}+ g(x)\leq \zeta(1+x) ~\mbox{ for $x>0$,}
 $$
where
$$
g(x)=\frac{1}{2}+\frac{x+1}{12}-\frac{(x+1)(x+2)(x+3)}{6!}=\frac{1}{6!}(414+49x-6x^2-x^3).
$$
In view of \cite[Lemma 10]{BR}, it suffices to show that $f(x)\leq g(x)$ on $(0,2)$.
For $0<x<2$,
$$g'(x)= \frac{1}{6!}(49-3x(x+4))>0,
$$
which shows that $g$ is increasing on $(0,2)$.
Since
$$f(2)=\frac{1}{3}\sqrt{\frac{2}{\pi}}<\frac{1}{3}<g(0)=\frac{23}{40},
$$
we have $f(x)\leq f(2)\leq g(0)\leq g(x)$ for all $0<x<2.$
This proves the claim for $0<x<2$, i.e., $1<s<3$.
In conclusion, the  inequality \eqref{eq-lem1} is verified for all $s>1$.
\end{proof}

\begin{figure}
\begin{center}
\includegraphics[width=12.0cm]{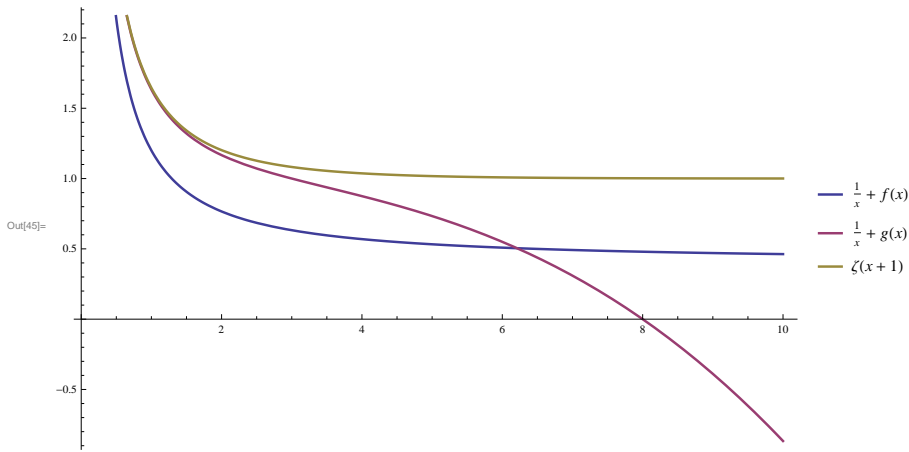}
\caption{The range for $x$ varies from $0.1$ to $10$\label{fig2}}
\end{center}
\end{figure}

\begin{remarque}\label{comparision}
Consider the functions $f$ and $g$ as in Lemma \ref{zetabound}. Thus, $\frac{1}{x}+ f(x)$ and $\frac{1}{x}+ g(x)$
 both forms a lower bound for $\zeta(1+x)$ for $x>0$. For $x>3$, we have
$$g'(x)= -\frac{1}{6!}(3x(x+4)-49)<0,
$$
which shows that $g$ is decreasing on $(3,\infty)$ and therefore, $g(x)\leq f(x)$ for all $x>s_2\approx 6.2102$,
where $s_2$ is the unique positive root of the equation given by $f(x)=g(x)$, i.e.,
 $$
 \left(\frac{x}{x+1}\right)\frac{1}{\sqrt{2\pi}}=\frac{1}{2}+\frac{x+1}{12}-\frac{(x+1)(x+2)(x+3)}{6!}.
 $$
It follows that Lemma \ref{zetabound} is an improved version of \cite[Lemma 10]{BR} for $x\geq s_2$.
For a quick comparison with the zeta function, in Figure \ref{fig2}, we have drawn the graphs of
$(1/x)+f(x)$, $(1/x)+g(x)$ and $\zeta (x+1) $.
 \end{remarque}

\begin{remarque}
Before seeing the work of \cite{BR},  we made use of a result of Lavrik \cite{Lav}: For $1<s<3$,
$$\zeta(s)-\frac{1}{s-1}-\gamma=\sum\limits_{n=1}^\infty \frac{\gamma_n}{n!} (s-1)^n,$$ where $\gamma$ is the Euler constant and  $|\gamma_n|\leq \frac{n!}{2^{n+1}}$. We thus obtained an alternative proof of \eqref{eq-lem1}.
\end{remarque}

\begin{lemme}\label{bound} If $s>1$, $i\geq 1$ is an integer, and $f(x)=\frac{(\log x)^{i}}{x^s}$, then one has
$$ \sum_{k=1}^\infty f(k)\leq \frac{i!}{(s-1)^{i}}\zeta(s).$$
\end{lemme}
\begin{proof}
 The function $f$ increases for $x\leq e^{i/s}$ and then decreases for $x\geq e^{i/s}$. By a simple change of
 variables, we have 
 $$I=\int_{1}^\infty f(x)dx= \frac{i!}{(s-1)^{i+1}}\cdot
 $$
 Let $N\geq 1$ be the integral part of $e^{i/s}$, so that $N\leq e^{i/s}<N+1$. Computations give, with help of
 Stirling's inequality $(i/e)^{i}\leq \frac{i!}{\sqrt{2\pi i}}$:
 $$
 \sum_{k=1}^{N-1} f(k)\leq \int_{1}^N f(x)dx
 $$
 and
$$
\sum_{k=N+2}^{\infty} f(k)\leq \int_{N+1}^\infty f(x)dx.
$$
It follows that
$$
\int_{N}^{N+1}f(x)dx \ge \left \{\begin{array}{ll}
f(N) &\mbox{ if $f(N)\leq f(N+1)$}\\
f(N+1)  &\mbox{ otherwise},
\end{array}
\right .
$$
and therefore,
$$
f(N)+f(N+1)-\int_{N}^{N+1}f(x)dx \leq f(e^{i/s})=\frac{(i/s)^i}{e^i}\leq \frac{i!}{\sqrt{2\pi i} s^i}\cdot
$$
From the above three inequalities, we get that
\beqq
 \sum_{k=1}^\infty f(k)&\leq& I+ f(e^{i/s})\\
&\leq& i!\left[\frac{1}{(s-1)^{i+1}}+ \frac{1}{\sqrt{2\pi i} s^i}\right]\\
&\leq& \frac{i!}{(s-1)^{i}}\left[\frac{1}{s-1} + \frac{1}{\sqrt{2\pi }}\Big(\frac{s-1}{s}\Big) \right]\\
&\leq& \frac{i!}{(s-1)^{i}} \zeta(s).
\eeqq
The third and the fourth inequalities follow from $\frac{s-1}{s}<1$ and Lemma \ref{zetabound}, respectively. This completes the proof of the lemma.
\end{proof}

%



Our next result provides bounds for the norm estimate of $C_\varphi$ on both sides.

\begin{theoreme}\label{main one} Let  $\varphi(s)=c_1+c_{2}2^{-s}  \hbox{\ with}\ {\rm Re}\, c_1\geq \frac{1}{2}+|c_2|$ and $c_2\neq 0$, thus inducing a bounded composition operator $C_\varphi:\,\mathcal{H}^2\to \mathcal{H}^2$. Then, we have
$$\zeta(2{\rm Re}\,  c_1)\leq \Vert C_\varphi\Vert ^2 \leq  \zeta(2{\rm Re}\,  c_1-r|c_2|),$$
where $r\leq 1$ is the smallest positive root of the quadratic polynomial
$$P(r)=|c_2| r^2+ (1-2{\rm Re}\,  c_1) r+ |c_2|.
$$
\end{theoreme}
\begin{remarque}
Observe that $P$ has two positive roots with product $1$, so one of them is less than or equal to $1$ (because $P(0)>0$ and $P(1)\leq 0$) and by our assumption
 $2{\rm Re}\,  c_1-r|c_2|\geq 2{\rm Re}\,  c_1-|c_2|\geq 1+ |c_2|>1$, so that $\zeta(2{\rm Re}\,  c_1-r|c_2|)$ is well defined.
\end{remarque}
\begin{proof}[Proof of Theorem {\rm \ref{main one}}] Without loss of generality, we can assume that $c_1$ and $c_2$ are positive.
Indeed, in  the general case, 
for $\varphi(s)=c_1+c_2 2^{-s}$, we set $c_1=\sigma_1+it_1$ and $c_2=|c_2|2^{i\varphi_2}$.
Note that ${\rm Re}\,  c_1=\sigma_1 >0$ by our assumption of the theorem.
Consider the two vertical translations $T_1$ and $T_2$ defined respectively by
$T_{1}(s)=s+it_1$ and $T_{2}(s)=s-i\varphi_2$, and set $\psi(s)=\sigma_1+|c_2|2^{-s}$. Then, one has
$\varphi=T_1 \circ \psi \circ T_2$ whence
$$ C_\varphi=C_{T_2} \circ C_\psi \circ C_{T_1}, 
$$
where   $C_{T_2}$ and $C_{T_1}$ are unitary operators.

Note that $C_{\varphi}(1)=1$. Now for $j>1$, we see that
$$C_{\varphi}(j^{-s})=j^{-c_1} \exp (-c_2 2^{-s} \log j)=j^{-c_1}\sum_{i=0}^\infty \frac{(-c_2 \log j)^{i}}{i!} (2^{i})^{-s}.
$$
In other terms, considering the orthonormal system $\{(2^{i})^{-s}\}_{i\geq 0}$ as the canonical basis of the range of $C_\varphi$ and the orthonormal system $\{j^{-s}\}_{j\geq 1}$ as the canonical basis of $\mathcal{H}^2$,  $C_\varphi$ can be viewed as the matrix $A=(a_{i,j})_{i\geq 0, j\geq 1}:\,\ell^{2}(\N)\to \ell^{2}(\N_0)$ with
$$
a_{i1}=
\left\{
\begin{array}{ll}
1 & \mbox{if } i=0 \\
0 & \mbox{if } i>0
\end{array},
\right.
$$
and
$$a_{i,j}=j^{-c_1} \frac{(-c_2 \log j)^{i}}{i!}~\mbox{ for $\ i\geq 0,\, j>1$.}
$$
By Theorem~A,  we already know that $A$ is bounded. We will give a direct proof of this fact, and moreover an upper and lower estimates of its norm. To that effect, we apply the Schur test with the following values of the parameters
$$ \alpha=1,\ \beta= \zeta(2 c_1-r c_2),\ p_j=j^{rc_2-c_1} \hbox{\ and}\ q_i=r^{i}.
$$

Now, we can check the assumptions of Schur's lemma.
Equality holds trivially in  the inequality (\ref{a})   for the case of $j=1$. For $j>1$,
$$\sum\limits_{i=0}^{\infty} |a_{i,j}| q_i=\sum\limits_{i=0}^{\infty} j^{-c_1} \frac{(c_2 \log j)^{i}}{i!} r^i=j^{rc_2-c_1}= \alpha p_j
$$
Thus, the inequality \eqref{a} is verified. Now, we verify the inequality \eqref{b}.  For the case $i=0$, we have
$$\sum\limits_{j=1}^{\infty} |a_{0,j}| p_j=\sum\limits_{j=1}^{\infty} j^{-(2c_1-rc_2)}=\zeta(2c_1-rc_2)\leq \beta q_0.
$$
 Finally, for $i\geq 1$, with the help of Lemma \ref{bound}, we have
$$\sum\limits_{j=1}^{\infty}|a_{i,j}|p_j=\frac{c_2^{i}}{i!}\sum\limits_{j=2}^{\infty}\frac{(\log j)^{i}}{j^{2c_1-rc_2}}\leq \frac{c_2^{i}}{i!}\frac{i!}{(2c_1-rc_2-1)^{i}}\zeta(2c_1-rc_2)=\beta q_i,
$$
where  $\frac{c_2}{2c_1-rc_2-1}=r$, that is, $P(r)=0$. The assumptions of the Schur lemma with the claimed values are thus verified, and the upper bound ensues.

 For the lower bound, we use reproducing kernels as usual (recall that $C_{\varphi}^{\ast}(K_a)=K_{\varphi(a)}$):
$$\Vert C_\varphi\Vert^2 \geq (S_{\varphi}^{\ast})^2:=\sup_{a\in \C_{1/2}} \frac{\Vert K_{\varphi(a)}\Vert^2}{\Vert K_a\Vert^2}= \sup_{a\in \C_{1/2}} \frac{\zeta(2{\rm Re}\, \varphi(a))}{\zeta(2{\rm Re}\,  a)}= \sup_{x>1/2} \frac{\zeta(2c_1-2c_2 2^{-x})}{\zeta(2 x)}\cdot
$$
The last equality in the above is obtained from basic trigonometry
and the fact that $\zeta(s)$ is a decreasing function on $(1,\infty)$. Now by letting $x\to \infty$, we get the lower bound for $\| C_\varphi\|$. 
\end{proof}

\begin{corollaire}\label{equal-case} Let  $\varphi(s)=c_1+c_{2}2^{-s}  \hbox{\ with}\ {\rm Re}\,  c_1= \frac{1}{2}+|c_2|$ and $c_2\neq 0$. Then, for
the inducing composition operator $C_\varphi:\,\mathcal{H}^2\to \mathcal{H}^2$, we have
$$\zeta(2{\rm Re}\, c_1)= \zeta(1+2|c_2|) \leq \Vert C_\varphi\Vert ^2 \leq \zeta(1+|c_2|)= \zeta(2{\rm Re}\,  c_1-|c_2|).
$$
\end{corollaire}
\begin{proof}
It suffices to observe that $r=1$ in Theorem \ref{main one} when ${\rm Re}\,  c_1= \frac{1}{2}+|c_2|$.
\end{proof}

\begin{remarque}
From the proof of Theorem \ref{main one}, it is evident that the lower bound of  $\Vert C_\varphi\Vert$ continues to hold for any
composition operator $C_\varphi$ with $c_0=0$ in \eqref{phi-expn}, namely, for any $\varphi(s)=\sum\limits_{n=1}^\infty c_n n^{-s}$.
\end{remarque}

\begin{remarque}
{\rm
\begin{itemize}
\item[(a)] Note that, if $c_2=0$, then $\varphi$ becomes a constant map and the induced  composition operator $C_\varphi$ is the evaluation map at $c_1$. Also it is known that
$$\Vert C_\varphi\Vert ^2=\zeta(2{\rm Re}\,  c_1).
$$
\item[(b)] Let $\varphi$ be a map as in \eqref{aff}. Then $C_\varphi$ cannot be a normal operator.
More generally, it cannot be a normaloid operator because,
$$
r(C_\varphi)=1 < \sqrt{\zeta(2{\rm Re}\,  c_1)} \leq \Vert C_\varphi\Vert.
$$ (see Remark \ref{spec-radius} and Theorem \ref{main one}).
\end{itemize}
}
\end{remarque}

\section{Approximation numbers}\label{MPQ3Sec4}

Recall that the $N^{th}$ approximation number $a_{N}(T),\ N=1,2,\ldots$, of an operator $T:\,H\to H$, where $H$ is a
Hilbert space, is the distance (for the operator norm) of $T$ to operators of rank $<N$. We refer to \cite{CAST}
for the definition and basic properties of those numbers. In the case $\varphi(z)=az+b$ on $H^2$ with $|a|+|b|\leq 1$,
Clifford and Dabkowski \cite{CLDA} computed exactly the approximation numbers $a_{N}(C_\varphi)$. In
the compact case $|a|+|b|<1$, they \cite{CLDA} showed in particular that
$$ a_{N}(C_\varphi)=|a|^{N-1}Q^{N-1/2} \quad \mbox{ for all } N\geq 1,
$$
where
$$Q=\frac{1+|a|^2-|b|^2-\sqrt{\Delta}}{2|a|^2}
$$
and where $\Delta>0$ is a discriminant depending on $a$ and $b$.

It is natural to ask whether we could get something similar for $\varphi(s)=c_1+c_2 2^{-s}$ and the associated
$C_\varphi$ acting on $\mathcal{H}^2$. We have here  the following upper bound, in which $2\,{\rm Re}\,  c_1-2|c_2|-1$ is
assumed to be positive which is indeed a necessary and sufficient condition for the compactness of $C_\varphi$.

\begin{theoreme}\label{approx. no}
Assume that $2\,{\rm Re}\,  c_1-2|c_2|-1>0$. Then the following exponential decay holds:
 $$a_{N+1}(C_\varphi)\leq  \sqrt{\frac{(2\, {\rm Re}\,  c_1 -1)(2\, {\rm Re}\,  c_1)}{(2\, {\rm Re}\,  c_1 -1)^2-(2|c_2|)^2}}  \left(\frac{2|c_2|}{2\,{\rm Re}\,  c_1-1}\right)^N.
 $$
\end{theoreme}
\begin{proof} Without loss of generality, we can assume that $c_1$ and $c_2$ are non-negative. Let $f(s)=\sum_{n=1}^\infty b_n n^{-s}\in \mathcal{H}^2$. Then
\beqq
C_{\varphi}f(s)&=&\sum_{n=1}^\infty b_n n^{-c_1} \exp(-c_2 2^{-s}\log n)\\
&=&\sum_{k=0}^\infty \frac{(-c_2)^k}{k!}\left (\sum_{n=1}^\infty b_n n^{-c_1}(\log n)^k\right ) 2^{-ks}.
\eeqq
Thus, designating by $R$ the operator of rank $\leq N$ defined by
$$Rf(s)=\sum_{k=0}^{N-1} \frac{(-c_2)^k}{k!}\left (\sum_{n=1}^\infty b_n n^{-c_1}(\log n)^k\right ) 2^{-ks},
$$
we obtain via the classical Cauchy-Schwarz inequality that
\beqq
\Vert C_{\varphi}(f)-R(f)\Vert^2 &= &\sum_{k=N}^{\infty} \frac{c_2^{2k}}{k!^{2}}\left |\sum_{n=1}^\infty b_n n^{-c_1}(\log n)^k\right|^2\\
&\leq& \sum_{k=N}^{\infty} \frac{c_2^{2k}}{k!^{2}}\left  (\sum_{n=1}^\infty |b_n|^2\right)\left (\sum_{ n=1}^\infty\frac{(\log n)^{2k}}{n^{2 c_1}}
\right).
\eeqq
 By Lemma \ref{bound}, the latter sum is nothing but
\beqq \sum_{ n=1}^\infty\frac{(\log n)^{2k}}{n^{2 c_1}} &=& \zeta^{(2k)}(2 c_1)\\
&\leq& \frac{(2k)!}{(2c_1-1)^{2k}}\zeta(2 c_1)\\
&\leq& \frac{(2k)! (2c_1)}{(2c_1-1)^{2k+1}}.
\eeqq
The last inequality follows by the simple fact that $\zeta(s)\leq \frac{s}{s-1}$ (see Lemma \ref{rely}).
Since
$$\sum_{n=1}^\infty |b_n|^2=\Vert f\Vert^2 ~\mbox{  and }~\frac{(2k)!}{(k!)^2} \leq \sum_{j=0}^{2k} {2k\choose j}=4^k,
$$
we get the following:
\beqq
\Vert C_\varphi -R\Vert^2&\leq& \sum_{k=N}^\infty \frac{c_{2}^{2k}}{(k!)^2}\frac{(2k)!(2 c_1)}{(2c_1-1)^{2k+1}}\\
&\leq &\sum_{k=N}^\infty \left (\frac{2c_2}{2c_{1}-1}\right)^{2k} \frac{2 c_1}{2c_{1}-1}\\
&=& \frac{2 c_1(2 c_1 -1)}{(2 c_1 -1)^2-(2 c_2)^2} \left(\frac{2 c_2}{2 c_1-1}\right)^{2N}.
\eeqq
Thus, we complete the proof.
\end{proof}
%

\section{Comments and questions}
\begin{itemize}

 \item  Is there a symbol $\varphi$ for which the strict inequalities
 $$\Vert C_\varphi\Vert>S_{\varphi}^{\ast}>S_\varphi
 $$ hold for $C_\varphi$ on $\mathcal{H}^2$? (refer to \cite{ABT} for similar problem in the case of classical Hardy space $H^{2}$).
  In the case $\varphi(s)=c_1+c_2 2^{-s}$, we probably have
 $$\Vert C_\varphi\Vert=S_{\varphi}^{\ast}=S_\varphi,
 $$  but this still needs a proof. Also observe that this $\varphi$ is not injective on $\C_{1/2}$.

 \item What can be said about $\Vert C_\varphi\Vert$ acting on $H^{2}(\Omega)$, where $\Omega$ is  the ball $\mathbb{B}_d$, or  the polydisk $\D^d$, when $\varphi(z)=A(z)+b$ with $A:\,\C^d\to \C^d$ a linear operator, i.e. when $\varphi$ is an affine map such that $\varphi(\Omega)\subset \Omega$? This might be difficult \cite{BA1}, but interesting.
 \end{itemize}

\subsection*{Acknowledgement}
The authors thank the referee for many useful comments.
The first author thanks the Council of Scientific and Industrial Research (CSIR), India,
for providing financial support in the form of a SPM Fellowship to carry out this research.
 The third author  thanks the Indian Statistical Institute of
Chennai for providing good and friendly working conditions in December 2015,  when this collaboration was initiated.


%
%
%
%
%
%
%
%

\end{document}